\documentclass{amsproc}
\usepackage{amssymb,amscd,verbatim,xypic,graphicx,amsthm}

\newtheorem{theorem}{Theorem}[section]

\theoremstyle{definition}
\newtheorem{definition}[theorem]{Definition}
\newtheorem{example}[theorem]{Example}
\newtheorem{corollary}[theorem]{Corollary}
\newtheorem{lemma}[theorem]{Lemma}
\newtheorem{proposition}[theorem]{Proposition}
\newtheorem{conjecture}[theorem]{Conjecture}
\theoremstyle{remark}
\newtheorem{remark}[theorem]{Remark}

\numberwithin{equation}{section}

\newcommand{\abs}[1]{\lvert#1\rvert}

\newcommand{\blankbox}[2]{%
  \parbox{\columnwidth}{\centering
    \setlength{\fboxsep}{0pt}%
    \fbox{\raisebox{0pt}[#2]{\hspace{#1}}}%
  }%
}

\usepackage[mathscr]{eucal}
\usepackage[dvips]{color}

\newcommand\ovl{\overline}

\newcommand \unb[2]{\underset{#1}{{\underbrace{#2}}}}

\newcommand \fk[1]{{{\mathfrak #1}}}
\newcommand \C[1]{{\mathcal #1}}

\newcommand \bb[1]{{\mathbb #1}}

\newcommand \wti[1]{{\widetilde {#1}}}

\newcommand \bC{{\bb C}}

\newcommand \bN{{\bb N}}
\newcommand \bR{{\bb R}}

\newcommand\one{1\!\!1}

\newcommand\CO{{\C O}}

\newcommand \vO{{\check \CO}}

\newcommand\LG{^L\!G}

\newcommand\ve{^\vee\!\!e}
\newcommand\vh{^\vee\!\!h}
\newcommand\vf{^\vee\!\!f}
\newcommand\vK{^\vee\!\!K}
\newcommand\vk{^\vee\!\!\mathfrak k}
\newcommand\vs{^\vee\!\!\mathfrak s}
\newcommand\vg{^\vee\!\!\mathfrak g}
\newcommand\vG{^\vee\!G}

\newcommand\eq{\begin{equation}}
\newcommand\eeq{\end{equation}}
\newcommand\eqn{\begin{equation*}}
\newcommand\eeqn{\end{equation*}}

\newcommand\ie{{\it i.e.~ }}

\newcommand\ep{{\epsilon}}
\newcommand\la{{\lambda}}

\newcommand\al{{\alpha}}
\newcommand\sig{{\sigma}}

\newcommand \Kt{\wti{K}}

\newcommand\vth{^\vee\!\theta}

\newcommand\Hom{\operatorname{Hom}}
\newcommand\Ad{\operatorname{Ad}}
\newcommand\ad{\operatorname{ad}}

\newcommand\Ind{\operatorname{Ind}}
\newcommand\Cas{\operatorname{Cas}}

\DeclareMathOperator{\im}{im}

\newcommand\bigwedgestar{{\textstyle{\bigwedge}^*}}

\begin{document}

\title[Dirac cohomology for complex groups]{Dirac cohomology and
  unipotent representations of complex groups} 

\author{Dan Barbasch}
\address{Department of Mathematics, Cornell University, Ithaca NY
  14850, USA}
%\curraddr{}
\thanks{The first author was supported in part by NSF Grants \#0967386
  and \#0901104.}

\author{Pavle Pand\v zi\'c}
\address{Dept. of Mathematics
                University of Zagreb, Bijeni\v cka 30, 10000 Zagreb, Croatia}   
\email{pandzic@math.hr}
\thanks{The second author was supported in part by a grant from the Ministry of Science, Education and Sports of the Republic of Croatia.}

%    General info
\subjclass{Primary 22E47 ; Secondary 22E46}
\date{June 1, 2010.}

\dedicatory{This paper is dedicated to Henri Moscovici.}

\keywords{Dirac cohomology, unipotent representations}

\begin{abstract}
This paper studies unitary representations with Dirac cohomology for
complex groups, in particular relations to unipotent representations
\end{abstract}

\maketitle

\section{Introduction}\label{sec:1}

In this paper we will study the problem of classifying unitary
representations with Dirac cohomology. We will focus on the case when the
group $G$ is a complex group viewed as a real group. It will easily follow that
a necessary condition for having nonzero Dirac cohomology is that twice the infinitesimal
character is regular and integral. The main conjecture is the following.

\bigskip
\begin{conjecture}\label{conj:main} Let $G$ be a complex Lie group viewed
as a real group, and $\pi$ be an irreducible unitary representation such that
twice the infinitesimal character of $\pi$ is regular and integral. Then
$\pi$ has nonzero Dirac cohomology if and only if $\pi$ is cohomologically induced
from an \textit{essentially unipotent} representation with nonzero Dirac cohomology.
Here by an essentially unipotent representation we mean a unipotent
  representation tensored with a unitary character

\end{conjecture}

We start with some background and motivation.

%\textbf{\small{NOTATION}}
%      \begin{enumerate}
Let  $G$ be the real points of a linear connected reductive  group.
Its Lie algebra will be denoted by $\frak g_0$. Fix a Cartan
involution $\theta$ and write $\frak g_0=\frak k_0 +\frak s_0$ for the
Cartan decomposition. Denote by $K$ the maximal compact subgroup of $G$ with
Lie algebra $\fk k_0.$ The complexification $\frak g:=(\frak g_0)_\bC,$
 decomposes as $\fk g=\fk k+\fk s.$

A representation $(\pi,\mathcal{H})$ on a Hilbert space is
called {\textit{unitary}}, if $\mathcal{H}$ admits a $G-$invariant
positive definite inner product. One of the major problems of
representation theory is to classify the irreducible unitarizable
modules of $G.$ As motivation for why this problem is important, we
present an example from automorphic forms. Let $\Gamma\subset G$ be a
discrete cocompact subgroup. A question of interest is the computation
of $H^*(\Gamma).$ Let $X:=\Gamma\backslash G/ K.$ Then
$H^i(\Gamma):=H^i_{top}(X,\bC),$ where $H^i_{top}(X,\bC)$ denotes the usual cohomology of
the topological space $X$: The theory of automorphic forms
provides insight into $H^i_{top}(X,\bC).$ A fundamental result of Gelfand
and Piatetsky-Shapiro is that
$$
L^2(\Gamma\backslash G)=\bigoplus
m_\pi \C H_\pi
$$
where $\pi$ are irreducible unitary representations of $G,$ and
$m_\pi<\infty$.  It implies that
\[
H^i(\Gamma)=H^i_{top}(X,\bC)=\bigoplus m_\pi H^i_{ct}(G,\C H_\pi)=\bigoplus m_\pi H^i(\fk g,
K;\C H_\pi).
\]
Here $H^i_{ct}(G,\C H_\pi)$ denotes the continuous cohomology groups (see \cite{BW}),
and the groups $H^i(\fk g,K;\C H_\pi)$ are the relative Lie algebra cohomology
groups defined in \cite{BW}, Chapter II, Section 6, or \cite{VZ}. Here the unitary representation
$\C H_\pi$ is replaced by the corresponding $(\fk g,K)$ module, denoted again by
$\C H_\pi$.

Thus to obtain information about $H^i(\Gamma)$ one needs to have
information about $m_\pi$ and $H^i(\fk g, K,\pi).$ 
It is very difficult to obtain information about the multiplicities
$m_\pi$.  On the other hand, knowledge about the vanishing of $H^i(\fk
g,K)$ for all unitary representations translates into vanishing of
$H^i(\Gamma).$ This approach leads one to consider the following problem.

\subsection*{Problem}
 Classify all irreducible admissible unitary modules
 with nonzero $(\fk g,K)$ cohomology.

\bigskip

A more general problem where $\bC$ is replaced by an arbitrary finite
dimensional representation was solved by Enright
\cite{E} for complex groups. Introducing more general coefficients has
the effect that $\C H_\pi$ is replaced by $\C H_\pi\otimes F^*$ for
some finite-dimensional representation $F$. The results were
generalized later by Vogan-Zuckermann \cite{VZ} to real groups as follows.
The $\la$ appearing below is such that the infinitesimal character of
$\C R_\fk q^s(\bC_\la)$ equals the infinitesimal character of $F$.
The answer is that $\pi=\C R_\fk q^s(\bC_\la)$, where
\begin{itemize}
\item[-] $\fk q=\fk l +\fk u\subset \fk g$ is a $\theta$ stable
  parabolic subalgebra,
\item[-] $\bC_\la$ is a unitary character of $\fk l,$
\item[-] $\C R_\fk q^s$ is \textit{cohomological induction}, and
  $s=\dim(\fk u\cap \fk k).$
\end{itemize}
The starting point for the proof is the fact $H^i(\fk g, K;\pi)\cong \Hom_K[\bigwedge^i\fk s,\pi].$
The reference \cite{BW} gives consequences of these results. For a survey of related more recent results, the reader may consult \cite{LS}.

\bigskip
A major role in providing an answer to the above problem is played by
the \textit{Dirac Inequality} of Parthasarathy \cite{P2}.
The adjoint representation of $K$ on $\fk s$ lifts to
$\Ad:\wti{K}\longrightarrow Spin(\fk s)$, where $\wti{K}$ is the spin double cover of $K$. The Dirac operator
$$
D:\C H_\pi\otimes Spin\longrightarrow \C
H_\pi\otimes Spin
$$
is defined as
\[
D=\sum_i b_i\otimes d_i \quad\in U(\fk g)\otimes C(\fk s),
\]
where $C(\fk s)$ denotes the Clifford algebra of $\fk s$ with respect to the Killing form,
$b_i$ is a basis of $\fk s$ and $d_i$ is the dual basis with respect to the Killing form,
and $Spin$ is a spin module for $C(\fk s)$.
$D$ is independent of the choice of the basis $b_i$ and $K-$invariant. It satisfies
\[
D^2=-(\Cas_\fk g\otimes 1 +\|\rho_\fk g\|^2)+(\Delta(\Cas_\fk k)+\|\rho_\fk k\|^2).
\]
In this formula, due to Parthasarathy \cite{P1},
\begin{itemize}
\item[-] $\Cas_\fk g$ and $\Cas_\fk k$ are the Casimir operators for $\fk g$ and
  $\fk k$ respectively,
\item[-]  $\fk h=\fk t +\fk a$ is a fundamental $\theta$-stable Cartan subalgebra
with compatible systems of positive roots for $(\fk g,\fk h)$ and $(\fk k,\fk t)$,
 \item[-] $\rho_\fk g$ and $\rho_\fk k$ are the corresponding half sums of positive roots,
\item[-] $\Delta:\fk k\to U(\fk g)\otimes C(\fk s)$ is given by $\Delta(X)=X\otimes 1+1\otimes\alpha(X)$,
where $\alpha$ is the action map $\fk k\to\fk s\fk o(\fk s)$ followed by the usual identifications
$\fk s\fk o(\fk s)\cong \bigwedge^2(\fk s)\hookrightarrow C(\fk s)$.
\end{itemize}
If $\pi$ is unitary, then $\C H_\pi\otimes Spin$ admits a
$K-$invariant inner product $\langle\ ,\ \rangle$ such that
$D$ is self adjoint with respect to this inner product. It follows that $D^2\geq 0$ on $\C H_\pi\otimes Spin$.
Using the above formula for $D^2$, we find that
\[
Cas_\fk g+||\rho_\fk g||^2\le Cas_{\Delta(\fk k)}+||\rho_\fk k||^2
\]
on any $K-$type $\tau$ occuring in $\C H_\pi\otimes Spin.$ Another
way of putting this is
\eq
\label{dirineq}
||\chi||^2\le ||\tau +\rho_\fk k||^2,
\eeq
for any $\tau$ occuring in $\C H_\pi\otimes Spin,$ where $\chi$ is the
infinitesimal character of $\pi.$ This is the Dirac inequality mentioned above.

These ideas are generalized by Vogan \cite{V} and Huang-Pand\v zi\'c \cite{HP1} as follows.
For an arbitrary admissible $(\fk g,K)$ module $\pi,$  we define {\it Dirac cohomology} of $\pi$ as
\[
H_D(\pi)=\ker D/(\ker D\cap \im D).
\]
Then $H_D(\pi)$ is a module for $\wti{K}$. If $\pi$ is unitary, $H_D(\pi)=\ker D=\ker D^2.$

The main result about $H_D$ is the following theorem conjectured by
Vogan.
\begin{theorem}\label{t:basic}
  \cite{HP1}
Assume that $H_D(\pi)$ is not zero, and let it contain an irreducible $\Kt$-module with highest weight $\tau$.
Let $\chi\in\fk h^*$ denote the infinitesimal character of $\pi$.
Then $w\chi=\tau +\rho_\fk k$ for some $w$ in the Weyl
group $W=W(\fk g,\fk h).$ More precisely, there is $w\in W$ such that
$w\chi\mid_\fk a=0$ and $w\chi\mid_\fk t=\tau +\rho_\fk k.$

Conversely, if $\pi$ is unitary and
$\tau=w\chi-\rho_\fk k$ is the highest weight of a $\Kt-$type occuring in $\pi\otimes Spin,$ then this
$\Kt-$type is contained in $H_D(\pi).$
\end{theorem}

This result might suggest that difficulties should arise in passing between $K$-types of $\pi$
and $\Kt$-types of $\pi\otimes Spin$. For unitary $\pi$, the situation
is however greatly simplified by the Dirac inequality.
Namely, together with (\ref{dirineq}), Theorem
\ref{t:basic} shows that the infinitesimal characters $\tau+\rho_\fk k$ of
$\Kt$-types in Dirac cohomology have minimal possible norm. This means that whenever such
$E(\tau)$ appears in the tensor product of a $K$-type $E(\mu)$ of $\pi$ and a $\Kt$-type $E(\sigma)$
of $Spin$, it necessarily appears as the PRV component \cite{PRV}, i.e.,
\eq
\label{PRV}
\tau = \mu+\sigma^-\qquad\text{up to } W_\fk k,
\eeq
where $\sigma^-$ denotes the lowest weight of $E(\sigma)$.

For unitary representations, the relation of Dirac cohomology to $(\fk g,K)$ cohomology is as follows.
(For more details, see \cite{HP1} and \cite{HKP}.)
One can write the $K$-module $\bigwedge(\fk s)$ as $Spin\otimes Spin$ if $\dim \fk s$ is even, or
twice the same space if $\dim\fk s$ is odd. It follows that
\eqn
\Hom_K(\bigwedge(\fk s),\pi\otimes F^*)=\Hom_{\Kt}(F\otimes Spin,\pi\otimes Spin),
\eeqn
or twice the same space if $\dim\fk s$ is odd. Since $D^2\geq 0$ on $\pi\otimes Spin$ and $D^2\leq 0$
on $F\otimes Spin$, it follows that
\eqn
H(\fk g,K;\pi\otimes F^*)=\Hom_{\Kt}(H_D(F),H_D(\pi)),
\eeqn
or twice the same space if $\dim\fk s$ is odd. In particular, if $\pi$ is unitary and it has nontrivial 
$(\fk g,K)$ cohomology, then $H_D(\pi)\ne 0.$

{For a representation to have nonzero $(\fk g, K)$ cohomology with coefficients
in a finite dimensional representation, the infinitesimal character
must be regular integral. Conversely, assume that $\pi$ is unitary with regular integral
infinitesimal character. Then the main result of \cite{SR} implies that $\pi$ is an $A_\fk q(\lambda)$-module,
and therefore it has nonzero $(\fk g, K)$ cohomology by the results of \cite{VZ}. (Hence it also has
nonzero Dirac cohomology, as explained above.)}

The hope is that unitary representations with Dirac cohomology will have
similarly nice properties. For $H_D(\pi),$ to be nonzero, Theorem
\ref{t:basic} provides a restriction on the infinitesimal character
$\chi_\pi$ which is weaker than regular integral. Namely, because
$\chi_\pi|_{\fk t}$ must be conjugate to $\tau+\rho_c$, it must be
regular integral for the roots in $\fk k.$ {Thus one expects to have representations with
nonzero Dirac cohomology with infinitesimal character that is not regular integral. Indeed,
we will describe many such examples in this paper. On the other hand, the conditions of
regularity and integrality with respect to $\fk k$ is still quite restrictive and we
cannot expect such representations to
capture the entire unitary dual.
The relatively few unitary representations
that have nonzero Dirac cohomology are however the borderline cases for unitarity in
the sense of Dirac inequality.

The paper is organized as follows. In Section 2 we first recall
  some well known facts about complex groups and their
  representations. Then we prove one of the main results of the paper,
  which says that a representation which is unitarily induced from a
  representation with $H_D\neq 0$ must have $H_D\neq 0$, provided
  twice its infinitesimal character is regular and integral.
In Section 3 we strengthen this result by actually calculating $H_D$ 
for representations induced from unitary characters whose
infinitesimal character is $\rho/2$. In Section 4 we generalize this
result to $GL(n,\bC)$ and more general infinitesimal characters
(we do not prove the full conjecture). 
Finally in Section 5 we discuss unipotent representations with
non-vanishing Dirac cohomology. In summary, the main general results are Theorem
\ref{thm:uind}, and Theorem \ref{thm:halfrho}. The other results of the paper
provide evidence for conjectures   \ref{conj:main}, \ref{conj:3.4}, and
\ref{conj:gl}. We plan to investigate the validity of these
conjectures in future papers. 

\smallskip
We dedicate this paper to Henri Moscovici. Henri introduced
the first author to the  beautiful theory of the heat kernel and index
theory on semisimple groups.

\section{Complex groups}
\label{sec:2.2}

\subsection{General setting}
Let $G$ be a complex group viewed as a real group. Let $H=TA$
be a $\theta-$stable Cartan subgroup with Lie algebra $\fk
h_0=\fk t_0 + \fk a_0$, a $\theta-$stable Cartan subalgebra. Let $B=HN$ be
a Borel subgroup. We identify $\fk h\cong\fk h_0\times \fk h_0,$ and
the complexifications
\begin{equation}
  \label{eq:cxcsg}
  \fk t\cong\{(x,-x)\ : x\in\fk h_0\},\qquad \fk a\cong\{ (x,x)\:\
  x\in\fk h_0\}.
\end{equation}

Admissible irreducible representations of $G$ are parametrized by
conjugacy classes of pairs $(\la_L,\la_R)\in\fk h_0\times\fk
h_0$ under the diagonal $\Delta(W)\subset W\times W$.
More precisely the following theorem holds.

Let $(\la_L,\la_R)$ be such that
$\mu:=\la_L-\la_R$ is integral. Write
$\nu:=\la_L+\la_R.$ We can view $\mu$ as a weight of $T$  and $\nu$ as
a character of $A.$ Let
\begin{equation*}
  X(\la_L,\la_R):=\Ind_B^G[\bC_\mu\otimes\bC_\nu\otimes\one]_{K-finite}.
\end{equation*}
\begin{theorem}[\cite{Zh}, \cite{PRV}]
The $K-$type with extremal weight $\mu$ occurs with multiplicity
1, so let $L(\la_L,\la_R)$ be the unique irreducible subquotient
containing this $K-$type.

\begin{enumerate}
\item Every irreducible admissible $(\fk g,K)$
module is of the form $L(\la_L,\la_R).$
\item Two such modules $L(\la_L,\la_R)$ and $L(\la_L',\la_R')$ are equivalent
if and only if the parameters are conjugate by $\Delta(W)\subset
W_c\cong W\times W.$ In other words, there is $w\in W$ such that
$w\mu=\mu'$ and $w\nu=\nu'.$
\item $L(\la_L,\la_R)$ admits a nondegenerate hermitian form if and
  only if there is $w\in W$ such that $w\mu=\mu,$ $w\nu=-\ovl{\nu}.$
\end{enumerate}
\end{theorem}

This result is a special case of the more general Langlands classification, which can be found
for example in the book \cite{Kn}.

We next describe the spin representation of the group $\Kt$.
Let $\rho:=\frac12\sum_{\al\in\Delta(\fk b,\fk h)} \al.$  Let $r$ denote the rank of $\fk g$.

\begin{lemma}
\label{spin}
The $Spin$ representation viewed as a $\Kt$-module is a direct sum of $[\frac{r}{2}]$ copies of
the irreducible representation $E(\rho)$ of $\Kt$ with highest weight $\rho.$
\end{lemma}

\begin{proof} The general description of the spin module is given for example in \cite{BW}, Lemma 6.9.
It says that the irreducible components of $Spin$ correspond to
choices of positive roots for $\fk g$ compatible with a fixed choice of positive roots for $\fk k$.
The multiplicity for each component is $[\frac{\dim\fk a}{2}]$.
In the complex case, there is only one such choice of positive roots, and $\dim\fk a$ is $r$.
\end{proof}

Lemma \ref{spin} implies that in calculating $H_D(\pi)$ for unitary $\pi$, one can replace $Spin$ by
$E(\rho)$ and then in the end simply multiply the result by multiplicity $[\frac{r}{2}]$.

By Theorem \ref{t:basic} and the above remark, a unitary representation $L(\la_L,\la_R)$
has Dirac cohomology if and only if there is $(w_1,w_2)\in W_c$ such that
\begin{equation}
  \label{eq:2.2.1}
  w_1\la_L+w_2\la_R=0,\qquad w_1\la_L-w_2\la_R=\tau +\rho
\end{equation}
where $\tau$ is the highest weight of a $\Kt-$type which occurs in $L(\la_L,\la_R)\otimes E(\rho).$
More precisely
\begin{equation}
  \label{eq:2.2.2}
 [H_D(\pi):E(\tau)]=\sum_\mu \,\,[\frac{r}{2}]\,\,[\pi:E(\mu)]\,\,[E(\mu)\otimes
  E(\rho):E(\tau)],
\end{equation}
where the sum is over all $K$-types $E(\mu)$ of $\pi$.

Write $\la:=\la_L.$ The first equation in (\ref{eq:2.2.1}) implies that
$\la_R=-w_2^{-1}w_1\la.$ The second one says that
$2w_1\la=\tau +\rho,$ so that $w_1\la$ must be regular, and $2w_1\la$
regular integral.  Replace $w_1\la$ by $\la.$
Thus we can write the  parameter of $\pi$ as
$(\la,-s\la)$ with $\la$ dominant, and $s\in W.$ Since
$L(\la,-s\la)$ is assumed unitary, it is hermitian. So there is $w\in
W$ such that
\begin{equation}
  \label{eq:cxherrmitian}
  w(\la+s\la)=\la + s\la,\quad w(\la-s\la)=-\la +s\la.
\end{equation}
This implies that $w\la=s\la,$ so $w=s$ since $\la$ is regular, and
$ws\la=s^2\la=\la$. So $s$ must be an involution.

Thus to compute $H_D(\pi)$ for $\pi$ that are unitary, we need
\begin{enumerate}
\item $L(\la,-s\la)$ that are unitary with
  \begin{equation}
    \label{eq:2.2.3}
 2\la=\tau+\rho,
  \end{equation}
in particular  $2\la$ is regular integral,
\item the multiplicity
  \begin{equation}
    \label{eq:2.2.4}
 \big[\ L(\la,-s\la)\otimes E(\rho)\ :\ E(\tau)\ \big].
  \end{equation}
\end{enumerate}

\subsection{Unitarily induced representations}
\label{sec:uind}
We consider the Dirac cohomology of a $\pi$ which is unitarily induced
from a unitary  representation of the Levi component of a
  parabolic subgroup $P=MN$ with nonzero Dirac cohomology. We denote the representation we are inducing
  from by $\pi_\fk m\otimes\bC_\xi$, where $\xi$ is a unitary character of $M$.

We choose a positive system $\Delta$ so that $\xi$ is dominant. Assume
that $P$ is such that $\fk p=\fk m +\fk n$ is the parabolic subalgebra determined by $\xi.$
Then $\Delta=\Delta_\fk m\cup\Delta(\fk n)$, where $\Delta_\fk m$ is a positive
root system for $\fk m$, while $\Delta(\fk n)$ denotes the set of roots such that the corresponding
root spaces are contained in $\fk n$.

The representation
$\pi_\fk m=L(\la_\fk m,-s\la_\fk m)$ satisfies
\begin{align}
&\la_\fk m +s\la_\fk m=\mu_\fk m,\quad &&2\la_\fk m=\mu_\fk m+\nu_\fk m,
\label{eq:uind1}\\
&\la_\fk m-s\la_\fk m=\nu_\fk m,\quad &&2s\la_\fk m=\mu_\fk m-\nu_\fk m.
\label{eq:uind2}
\end{align}
By assumption, $\pi_\fk m$ has Dirac cohomology. So
\begin{equation}
  \label{eq:uind3}
2\la_\fk m=\mu_\fk m +\nu_\fk m=\tau_\fk m+\rho_\fk m.
\end{equation}
Here $\rho_\fk m$ is the half sum of the roots in $\Delta_\fk m$ and $\tau_\fk m$
is dominant with respect to $\Delta_\fk m$. Also,
\begin{equation}
  \label{eq:uind7}
\big[\pi_\fk m\otimes F(\rho_\fk m)\ :\ F(\tau_\fk m)\big] \ne 0.
\end{equation}
Here and in the following $F(\chi)$ denotes the finite-dimensional $\fk m$-module with highest weight $\chi$,
to distinguish it from $E(\chi)$ which denotes the finite-dimensional $\fk g$-module with highest weight $\chi$.
We are also going to use analogous notation when $\chi$ is not necessarily dominant, but any extremal weight of
the corresponding module.

The induced module $\pi=Ind_P^G [\pi_\fk m\otimes\xi]$
%(assumed irreducible because of a
%classification theorem for the unitary dual at such infinitesimal character)
has parameters
\begin{equation}
  \label{eq:uind4}
  \begin{aligned}
    &\la=\xi/2+\la_\fk m,\quad &&\mu=\xi +\mu_\fk m,\\
    &s\la=\xi/2+s\la_\fk m,\quad &&\nu=\nu_\fk m.
  \end{aligned}
\end{equation}
In order to have Dirac cohomology, $2\la$ must be regular integral; so
assume $\xi$ is such that this is the case. Let $\Delta'$  be the positive system such that $\la$ is
dominant. Then
\begin{equation}
  \label{eq:uind4a}
  2\la=\xi +2\la_\fk m=\xi +\mu_\fk m +\nu_\fk m=\tau'+\rho'.
\end{equation}
Here $\rho'$ is the half sum of the roots in $\Delta'$, and $\tau'$
is dominant with respect to $\Delta'$.
In order to see that $\pi$ has nonzero Dirac cohomology, we need

\begin{lemma}
\label{resErho}
The restriction of the $\fk g$-module $E(\rho)$ to $\fk m$ is isomorphic to $F(\rho_\fk m)\otimes \bC_{-\rho_\fk n}\otimes \bigwedge^*\fk n$,
where $F(\rho_\fk m)$ denotes the irreducible $\fk m$-module with highest weight $\rho_\fk m$ and $\rho_\fk n$ denotes the half sum of roots
in $\Delta(\fk n)$.
\end{lemma}
\begin{proof}
Since $\fk g$ and $\fk m$ have the same rank, we can use Lemma \ref{spin} to replace $E(\rho)$ and $F(\rho_\fk m)$ by the corresponding spin modules.
Recall that the spin module $Spin_\fk m$ can be constructed as $\bigwedge^*\fk m^+$, where $\fk m^+$ is a maximal isotropic subspace of $\fk m$.
We can choose $\fk m^+$ so that it contains all the positive root subspaces for $\fk m$, as well as a maximal isotropic subspace $\fk h^+$ of the Cartan subalgebra $\fk h$. To construct $Spin_\fk g$, we can use the maximal isotropic subspace $\fk g^+=\fk m^+\oplus\fk n$ of $\fk g$.
It follows that $Spin_\fk g=Spin_\fk m\otimes \bC_{-\rho_\fk n}\otimes \bigwedge^*\fk n$. The $\rho$-shift comes from the fact that the highest weight
of $Spin_\fk m$ is $\rho_\fk m$ and the highest weight of $Spin_\fk g$ is $\rho$, while the highest weight of $Spin_\fk m\otimes \bigwedge^*\fk n$
is $\rho_\fk m+2\rho_\fk n=\rho+\rho_\fk n$.
\end{proof}

Since $\pi$ is unitary, the computation for its Dirac cohomology is
\begin{equation}
  \label{eq:uind5}
\begin{aligned}
&\big[\pi\otimes E(\rho)\ :\ E(\tau')\big]=
\big[\pi_\fk m\otimes\bC_\xi\otimes E(-\tau')\mid_\fk m\ :\ E(\rho)\mid_\fk m\big]=\\
&\big[\pi_\fk m\otimes\bC_\xi\otimes E(-\tau')\mid_\fk m
\ :\ F(\rho_\fk m)\otimes\bC_{-\rho_\fk n}\otimes\bigwedgestar\fk n\big]=\\
&\big[\bC_{\xi+\rho_\fk n}\otimes\pi_\fk m\otimes
F(\rho_\fk m)\otimes E(-\tau')\mid_\fk m \ :\ \bigwedgestar\fk n\big].
\end{aligned}
\end{equation}
Here the first equality used Frobenius reciprocity, while the second equality used Lemma \ref{resErho}.
Note that the dual of $E(\tau')$ is the module $E(-\tau')$ which has lowest weight $-\tau'$ with respect to
$\Delta'$.

Using (\ref{eq:uind4a}) and (\ref{eq:uind3}), we can write
\begin{equation}
  \label{eq:uind6}
-\tau'=-2\la+\rho'=-\xi -\mu_\fk m-\nu_\fk m+\rho'=-\xi -\tau_\fk m
-\rho_\fk m+\rho'.
\end{equation}
The positive system $\Delta\supset\Delta_\fk m$ was chosen so that
$\xi$ is dominant, and $2\la_\fk m$ was dominant for $\Delta(\fk m).$
Thus $\Delta_\fk m\subset \Delta,\Delta'$.
Because of (\ref{eq:uind7}), the LHS of the last line of (\ref{eq:uind5}) contains the
representation

\eqn
\bC_{\xi+\rho_\fk n}\otimes F(\tau_\fk m)\otimes E(-\tau')\mid_\fk m  \supseteq
\bC_{\xi+\rho_\fk n}\otimes F(\tau_\fk m-\tau').
\eeqn
Namely, $F(\tau_\fk m-\tau')$ is the PRV component of $F(\tau_\fk m)\otimes F(-\tau')\subseteq F(\tau_\fk m)\otimes E(-\tau')\mid_\fk m$.
By (\ref{eq:uind4a}) and (\ref{eq:uind3}), $\tau_\fk m-\tau'=-\xi-\rho_\fk m+\rho'$, so
\eqn
\bC_{\xi+\rho_\fk n}\otimes F(\tau_\fk m-\tau')\supseteq F(\rho_\fk n-\rho_\fk m+\rho')=F(w_\fk m\rho+\rho'),
\eeqn
where $w_\fk m$ is the longest element of the Weyl group of $\fk m$. Namely, $w_\fk m$ sends all roots in $\Delta_\fk m$ to negative roots for
$\fk m$, while permuting the roots in $\Delta(\fk n)$, so $w_\fk m\rho=-\rho_\fk m+\rho_\fk n$.

So we see that the LHS of the last line of (\ref{eq:uind5}) contains the $\fk m$-module $F(w_\fk m\rho+\rho')=F(w_\fk m\rho'+\rho).$ Namely,
both $w_\fk m\rho+\rho'$ and $w_\fk m\rho'+\rho=w_\fk m(w_\fk m\rho+\rho')$ are extremal weights for the same module.

We will show that
\begin{equation}
  \label{eq:hrho}
\big[ F(w_\fk m\rho'+\rho):\bigwedgestar\fk n\big]\ne 0.
\end{equation}
This will prove that (\ref{eq:uind5}) is nonzero, and consequently that $\pi$ has nonzero Dirac cohomology.

Note that $w_\fk m\rho'+\rho$ is a sum of roots in $\Delta(\fk n),$
and antidominant for $\Delta_\fk m,$ because for any simple $\gamma\in\Delta_\fk m$,
$\langle\rho',\check\gamma\rangle\in \bN^+$ and
$\langle\rho,\check\gamma\rangle=1$. Moreover,
\begin{equation}
  \label{eq:hrho8}
  w_\fk m\rho'+\rho=\sum_{\langle\al,w_\fk m\rho'\rangle >0,\
    \langle\al,\rho\rangle >0} \al.
\end{equation}

To show that (\ref{eq:hrho}) holds, it is enough to show that
\eq
\label{lowvector}
v:=\bigwedge_{\langle\al,\rho\rangle
>0,\ \langle\al, w_\fk m\rho'\rangle >0}\ e_\al\quad\in\bigwedgestar\fk n
\eeq
is a lowest weight vector
for $\Delta_\fk m$. Here $e_\al$ denotes a root vector for the root $\alpha$.

Let $\gamma\in\Delta_\fk m.$ Then, up to constant factors,
\begin{equation}
  \label{eq:hrho9}
  \ad e_{-\gamma}e_\al=
  \begin{cases}
    0 &\text{ if } \al-\gamma \text{ is not a root,}\\
    e_{-\gamma +\al} &\text{ if } \al-\gamma \text{ is a root.}\\
  \end{cases}
\end{equation}
But $\langle -\gamma,w_\fk m\rho'\rangle >0$, and
$\langle\al,w_\fk m \rho'\rangle >0$ by assumption, so
\begin{equation}
  \label{eq:hrho10}
  \langle-\gamma+\al,w_\fk m\rho'\rangle > 0+0=0.
\end{equation}
Also, if $-\gamma+\alpha$ is a root, then it is in $\Delta(\fk n)$, since $\alpha\in\Delta(\fk n)$
and $\fk n$ is an $\fk m$-module. So $\langle-\gamma+\al,\rho\rangle > 0$.
Thus every $e_{-\gamma +\al}$ appearing in (\ref{eq:hrho9}) is one of the factors in (\ref{lowvector}).

The claim now follows from the formula
\begin{equation}
  \label{eq:hrho11}
  \ad e_{-\gamma}\bigwedge e_\al=\sum e_{\al_1}\wedge \dots \wedge
  \ad e_{-\gamma}e_{\al_i}\wedge \dots .
\end{equation}
In each summand either $\ad e_{-\gamma}e_{\al_i}$ equals 0, or is a multiple of one of the
root vectors already occurring in the same summand. So $\ad e_{-\gamma} v=0.$
We have proved
\begin{theorem}
\label{thm:uind}
Let $P=MN$ be a parabolic subalgebra of $G$ and let $\Delta=\Delta_\fk m\cup\Delta(\fk n)$ be the corresponding system
of positive roots. Let $\pi_\fk m$ be an irreducible unitary representation of $M$ with nonzero Dirac cohomology, and let
$\xi$ be a unitary character of $M$ which is dominant with respect to $\Delta$. Suppose that twice the infinitesimal character of $\pi=Ind_P^G [\pi_\fk m\otimes\xi]$ is regular and integral. Then $\pi$ has nonzero Dirac cohomology.\qed
\end{theorem}

\begin{example}
\label{ex:sphalfrho}
Let $\fk g:=sp(10),$ and take infinitesimal
character $\rho/2$, which is conjugate to
\begin{equation*}
(2,1,5/2,3/2,1/2).
\end{equation*}
According to \cite{B}, the spherical representation is
not unitary, but the parameter
\begin{equation}
  \label{eq:uind8}
  (2,1;1/2,5/2,3/2)\times (-1,-2;-1/2,5/2,3/2)
\end{equation}
which  has $\mu=(3,3,1,0,0)$ and $\nu=(1,-1;0,5,3),$  is unitary
because it is unitarily induced from a representation on $GL(2)\times
Sp(6)$ which is the trivial on $GL(2)$ and the nonspherical component
of the metaplectic representation on $Sp(6)$ (see below).

 The Spin representation is a multiple of $E(\rho)=E(5,4,3,2,1).$ So the multiplicity
 $\big[ L(\la,-s\la)\ :\ E(\rho)\big]$ has a chance to be nonzero since the sums of coordinates
 in $\mu$ and $\rho$ have the same parity. $L(\la,-s\la)$  is
 \textbf{not} unitarily
 induced from a unitary character. The coordinates of $\nu$
 corresponding to the $0'$s in the coordinates of $\mu$ are $5,3$; they
 would have had to have been $4,2.$ Rather, $L(\la,-s\la)$ is unitarily induced from
\eqn
\fk m=\fk m_1\times\fk m_2=gl(2)\times sp(6),
\eeqn
with a character on the $gl(2)$ and one of the
metaplectic representations on $sp(6).$ The parameter of the metaplectic representation is
\begin{equation}
  \label{eq:uind9}
  \la_{\fk m_2}=(1/2,5/2,3/2),\qquad -s\la_{\fk m_2}=(-1/2,5/2,3/2)
\end{equation}
with $\mu_{\fk m_2}=(1,0,0)$ and $\nu_{\fk m_2}=(0,5,3).$ Its $K-$structure is
$(1+2k,0,0).$ Then
\eqn
\tau'_{\fk m_2}=2\la_{\fk m_2}-\rho'_{\fk m_2}=(1,5,3)-(1,3,2)=(0,2,1).
\eeqn
The character of $\fk m_1=gl(2)$ is
$(3,3)$, and we can view it as the character $\xi=(3,3,0,0,0)$ of $\fk m$.

Let us change the parameter in (\ref{eq:uind9}) to
 \begin{equation}
  \label{eq:uind10}
  \la_{\fk m_2}=(5/2,3,2,1/2),\qquad -s\la_{\fk m_2}=(5/2,3/2,-1/2)
\end{equation}
so that $\Delta_{\fk m_2}$ becomes the usual positive system. This changes
$\rho'=(4,2,1,5,3)$ into $(4,2,5,3,1)$. One easily checks that
$\rho_\fk n=(9/2,9/2,0,0,0)$. Thus the last line of (\ref{eq:uind5}) becomes
\begin{equation}
  \label{eq:uind11}
  \big[\bC_{(15/2,15/2,0,0,0)}\otimes F(0,0,1+2k,0,0)\otimes
  F(1/2,-1/2,3,2,1)\ :\ \bigwedgestar\fk n\big].
\end{equation}
The LHS contains the representation with lowest weight conjugate to
\eqn
w_\fk m\rho' +\rho=(2,4,-5,-3,-1)+(5,4,3,2,1)=(7,8,-2,-1,0).
\eeqn
This is the sum of the following roots in $\Delta(\fk n):$
\begin{equation}
  \label{eq:uind12}
  \begin{aligned}
&2\ep_1,\ep_1+\ep_2,\ep_1-\ep_3,\ep_1-\ep_4,\ep_1\pm\ep_5,\\
&2\ep_2,\ep_2-\ep_3,\ep_2\pm\ep_4,\ep_2\pm\ep_5.
  \end{aligned}
\end{equation}
This set of roots is stable under the operation of adding negative simple $\fk m$-roots, \ie,
$-\ep_1+\ep_2,-\ep_3+\ep_4,-\ep_4+\ep_5,-2\ep_5.$ Thus the vector $\bigwedge e_\al$ is
a lowest weight vector for $\Delta_\fk m$. So (\ref{eq:uind11}) is not 0.
\end{example}
%%%%%%%%%%%%%%%%%%%%%%%%%%%%

\section{Infinitesimal character ${\rho/2}$}\label{sec:rhoovertwo}

This case is the smallest possible in view of the necessary condition (\ref{eq:2.2.3}),
and thus it warrants special attention.
In this case equation (\ref{eq:2.2.3}) becomes
\begin{equation}
  \label{eq:rho/2}
  2(\rho/2)=\tau + \rho,
\end{equation}
so $\tau=0.$ Then the multiplicity in (\ref{eq:2.2.4}) becomes
\begin{equation}
  \label{eq:multrho/2}
  [L(\rho/2,-s\rho/2)\ :\ E(\rho)].
\end{equation}

\subsection
{Induced from a unitary character, infinitesimal character $\rho/2$}
\label{sec:halfrho}

We look at the special case when $\pi$ has infinitesimal character
$\rho/2,$ and is unitarily induced from a unitary character $\xi$ on a Levi
component $\fk m.$ In this case we will be able to improve over the result of \ref{sec:uind}.

Choose a positive system $\Delta$ so that $\xi$
is dominant, and let $\fk p=\fk m +\fk n$ be the parabolic
subalgebra determined by $\xi.$ The representation  $\pi=L(\la,-s\la)$
satisfies
\begin{align}
&\la +s\la=\xi,\quad &&2\la=\xi+2\rho_\fk m,  \label{eq:hrho1}\\
&\la-s\la=2\rho_\fk m,\quad &&2s\la=\xi-2\rho_\fk m.\label{eq:hrho2}
\end{align}
It can be shown that this implies $s=w_\fk m,$ the long Weyl group element in $W(\fk m).$
This fact is however not needed in the following.

Let $\Delta'$ be a positive root system so that $2\lambda$ is dominant. Then
$2\la=\rho'$ and $\tau'=0.$ Thus
$$
\xi=\rho'-2\rho_\fk m,\qquad s\rho'=\rho'-4\rho_\fk m.
$$
Next, the formula (\ref{eq:uind5}) for the case of general $\la$ simplifies to
\begin{equation}
  \label{eq:hrho5}
\begin{aligned}
&\big[\pi:\ E(\rho)\big]=
\big[\bC_\xi :\ E(\rho)|_\fk m\big]=
\big[\bC_\xi:\ F(\rho_\fk m)\otimes
\bC_{-\rho_\fk n}\otimes\bigwedgestar\fk n\big]=\\
&\big[\bC_\xi\otimes F(\rho_\fk m)\otimes
\bC_{\rho_\fk n}\ :\ \bigwedgestar\fk n\big].
\end{aligned}
\end{equation}
The LHS of the last line of (\ref{eq:hrho5}) has highest weight
\begin{equation}
  \label{eq:hrho6}
 \xi+\rho_\fk m +\rho_\fk n=\rho'-2\rho_\fk m +\rho_\fk m +\rho_\fk
n=\rho'+w_\fk m\rho,
\end{equation}
and lowest weight
\begin{equation}
  \label{eq:hrho7}
w_\fk m(\rho'+w_\fk m\rho)=w_\fk m\rho'+\rho.
\end{equation}
We have already shown in Subsection \ref{sec:uind} that (\ref{eq:hrho5}) is nonzero;
  we now show that it is equal to 1, i.e., that the multiplicity of $F(w_\fk m\rho' +\rho)$
   in $\bigwedgestar\fk n$ is equal to 1. We are going to use some classical results of Kostant
   \cite{K1}, \cite{K2} which we describe in the following.

   If $B\subset \Delta,$ denote by
  $2\rho(B)$ the sum of roots in $B.$ In this notation,
  \begin{equation}
    \label{eq:hrho12}
    \rho +w_\fk m\rho'=2\rho(B),
  \end{equation}
where $B:=\big\{ \al\in\Delta (\fk n)\ :\ \langle\rho',\al\rangle >0\big\}$.

\begin{lemma}[Kostant]
\label{Kostantlemma}
Let $B\subset\Delta$ be arbitrary, and denote by $B^c$ the
complement of $B$ in $\Delta.$ Then
\begin{equation*}
  \langle 2\rho(B),2\rho(B^c)\rangle \ge 0,
\end{equation*}
with equality if and only if there is $w\in W$ such that
$2\rho(B)=\rho+w\rho.$ In that case, $B$ is uniquely determined by $w$ as $B=\Delta\cap w\Delta$.
\end{lemma}
\begin{proof}
Since $2\rho(B)+2\rho(B^c)=2\rho$, we have
\begin{equation}
  \label{eq:hrho13}
  \langle 2\rho(B),2\rho(B^c)\rangle=\langle 2\rho(B),2\rho-2\rho(B)\rangle =
\langle\rho,\rho\rangle -\langle\rho-2\rho(B),\rho-2\rho(B)\rangle.
\end{equation}
But $\rho-2\rho(B)$ is a weight of $E(\rho),$ so the expression in
(\ref{eq:hrho13}) is indeed $\ge 0.$ It is equal to $0$ precisely when
$\rho-2\rho(B)$ is an extremal weight of $E(\rho)$. In that case it is conjugate to the
lowest weight $-\rho$, i.e., there is $w\in
W$ such that $\rho-2\rho(B)=-w\rho.$

For the last statement, notice that $\Delta=(\Delta\cap w\Delta)\cup (\Delta\cap -w\Delta)$, and that consequently
$\rho+w\rho=2\rho(\Delta\cap w\Delta)$, since the elements of $\Delta\cap -w\Delta$ cancel out in the sum $\rho+w\rho$.
\end{proof}

\begin{corollary}
  The weight $\rho +w_\fk m\rho'$ occurs with multiplicity 1 in
  $\bigwedgestar\fk n.$
\end{corollary}
\begin{proof}
We can write $w_\fk m\rho'=x\rho$ for a unique $x\in W.$ By the last statement of Lemma \ref{Kostantlemma}, it follows that the set
$B$ from (\ref{eq:hrho12}) is uniquely determined, and hence the corresponding multiplicity is one.
\end{proof}
We have proved

\begin{theorem}
\label{thm:halfrho}
Let $P=MN$ ba a parabolic subalgebra of $G$, and let $\Delta$ be the set
of positive roots corresponding to $P$. Assume that $\pi$ is a representation of $G$ with
infinitesimal character $\rho/2$, which is unitarily induced from a character $\xi$ of $M$,
such that $\xi$ is dominant with respect to $\Delta$. Then the Dirac cohomology of $\pi$ consists
of the trivial $K$-module with multiplicity $[Spin:E(\rho)]$.
\end{theorem}

We give some applications of this theorem in the next
  sections. In particular we collect evidence for the following
  conjecture.

  \begin{conjecture}\label{conj:3.4}
    A unitary representation with infinitesimal character $\rho/2$ has
    Dirac cohomology consisting of the trivial $K-$type with
    multiplicity $[Spin:E(\rho)]$ or 0.
  \end{conjecture}

This conjecture sharpens the main conjecture in the introduction for
the special case of infinitesimal character $\rho/2$, in
the sense that it predicts the size of $H_D(\pi)$ precisely in case when it is nonzero.

\subsection{Type A}
In this case,
\begin{equation}
  \label{eq:Arho}
\la= (\frac{n-1}{4},\dots ,\frac{-n+1}{4}).
\end{equation}
\begin{comment}
and $\vg(\la)\cong A(\left[\frac{n+1}{2}\right])\times
A(\left[\frac{n}{2}\right])$.
\end{comment}
By the classification of unitary representations from \cite{V1}, the irreducible unitary
representations with infinitesimal character $\rho/2$ are all unitarily induced
irreducible from unitary characters on Levi components.
Therefore, this case is covered by Theorem \ref{thm:halfrho}.
It follows that any irreducible unitary representation $\pi$ with infinitesimal
  character $\rho/2$ has Dirac cohomology consisting of
the trivial $K-$type occuring with multiplicity $[Spin:E(\rho)]$.

\subsection{Type B}

In this case
\begin{equation}
  \label{eq:Brho}
  \la= (\frac{2n-1}{4},\dots ,\frac14).
\end{equation}
The coroots integral on $\la$ form a subsystem of type $A(n).$ In the
notation of Section \ref{unip},  $\vg(\la)\cong A(n)$.
According to \cite{B}, the unitary representations are all unitarily
induced irreducible from unitary characters on Levi
components. Thus Theorem \ref{thm:halfrho} applies. It follows that any
irreducible unitary representation $\pi$ with infinitesimal
character $\rho/2$ has Dirac cohomology consisting of
the trivial $K-$type occuring with multiplicity $[Spin:E(\rho)]$.

\subsection{Type C}
In this case
\begin{equation}
  \label{eq:Crho}
  \la= (\frac{n}{2},\dots ,\frac12),
\end{equation}
and in the notation of Section 5, $\vg(\la)\cong B(\left[\frac{n}{2}\right])\times
D(\left[\frac{n+1}{2}\right])$ or $B(\left[\frac{n+1}{2}\right])\times
D(\left[\frac{n}{2}\right])$ depending on the parity of $n.$

According to \cite{B}, any irreducible unitary representation must be unitarily induced irreducible from
unitary characters on factors of type A of a Levi component, and
trivial or metaplectic representation on the factor of type C.
The latter case is not covered by Theorem \ref{thm:halfrho}. This situation was illustrated in Example
\ref{ex:sphalfrho}. 

The metaplectic representations will be analyzed in Section \ref{unip}. In particular, we will see
that depending on the parity of $n$, exactly one of the two metaplectic representations has nonzero Dirac cohomology. 

Our conjectures predict that the Dirac cohomology of an irreducible unitary representation $\pi$ with infinitesimal
character $\rho/2$ is the trivial $K-$type occuring with multiplicity $[Spin:E(\rho)]$, in case the representation that 
$\pi$ is induced from
has the metaplectic representation of appropriate parity, or the trivial representation on the factor of type C of the Levi component. Otherwise, $H_D$ has to
be zero.

\subsection{Type D}
In this case
\begin{equation}
  \label{eq:Drho}
  \la= (\frac{n-1}{2},\dots ,\frac12,0),
\end{equation}
and in the notation of Section 5, $\vg(\la)\cong D(\left[\frac{n+1}{2}\right])\times
D(\left[\frac{n}{2}\right])$.
According to \cite{B}, the unitary representations are all unitarily induced irreducible from
unitary characters on the factors of a Levi component of type A, and
a unipotent representation of the factor of the Levi component of
type D.

Our conjectures predict that such a representation $\pi$ has nonzero Dirac cohomology precisely when
the corresponding unipotent representation has nonzero Dirac cohomology and that in this
case $H_D(\pi)$ consists of the trivial $K-$type occuring with multiplicity $[Spin:E(\rho)]$.
We describe the unipotent representations with nonzero Dirac cohomology in Section \ref{unip}.

\subsection{Type F}
We investigate Conjecture \ref{conj:3.4}. The calculations were performed using LiE.
We list the hermitian parameters in the case of infinitesimal
character $\rho/2$ in the case of $F_4.$ The simple roots, coroots and
weights are

\begin{align}
&\begin{matrix}
    (0,1,-1,0)&-&(0,0,1,-1)&=>=&(0,0,0,1)&-&(1/2,-1/2,-1/2,-1/2).
  \end{matrix}  \label{eq:f4sr}\\
&\begin{matrix}
    (0,1,-1,0)&-&(0,0,1,-1)&=<=&(0,0,0,2)&-&(1,-1,-1,-1),
  \end{matrix}  \label{eq:f4cr}\\
&\begin{matrix}
    (1,1,0,0)&-&(2,1,1,0)&=>=&(3/2,1/2,1/2,1/2)&-&(1,0,0,0).
  \end{matrix}
  \label{eq:f4wt}
\end{align}
In these coordinates, $\rho/2=(5/2,3/2,1,1/2).$ The hermitian parameters
are in the following list. The $K-$type $\mu'$ indicates a $K-$type
which has signature opposite to that of the lowest $K-$type $\mu.$
For the parameters where the coordinates are $(\dots,1,\dots)\times
(\dots,-1,\dots),$ the Langlands quotients are unitarily induced
irreducible from the remainder of the parameter on a $B_3$; so the
representation is unitary if and only if the one with remainder on
$B_3$ is unitary. So we did not list a $\mu'$ which detects the
nonunitarity.
It is visible from the table that all irreducible unitary
representations with infinitesimal character $\rho/2$ have Dirac cohomology
consisting of the trivial $K-$type occuring with multiplicity $[Spin:E(\rho)]$.
Namely, each of these representations has $K-$type $E(\rho)$ with
multiplicity one. All unitary representations are unitarily induced
from unipotent representations tensored with unitary characters. The
results conform to Conjecture \ref{conj:3.4}.

\newpage
\begin{tabular}{|ll|l|l|l|l|l|}
$(5/2,3/2,1,1/2)\times$&$\mu$ & $\nu$ &Unitary& $E(\rho)$ & $\mu'$\\
&&&&&\\
$(5/2,3/2,1,1/2)$&$(0,0,0,0)$&$(5,3,2,1)$&NO&&$(1,0,0,0)$\\
$(5/2,3/2,1,-1/2)$&$(1,0,0,0)$&$(0,5,3,2)$&NO&&$(1,1,0,0)$\\
$(5/2,-3/2,1,1/2)$&$(3,0,0,0)$&$(0,5,2,1)$&NO&&$(3,1,0,0)$\\
$(-5/2,3/2,1,1/2)$&$(5,0,0,0)$&$(0,3,2,1)$&YES&1&\\
$(5/2,-3/2,1,-1/2)$&$(3,1,0,0)$&$(0,0,5,2)$&NO&&$(3,1,1,1)$\\
$(-5/2,3/2,1,-1/2)$&$(5,1,0,0)$&$(0,0,3,2)$&NO&&$(5,1,1,1)$\\
$(-5/2,-3/2,1,1/2)$&$(5,3,0,0)$&$(0,0,2,1)$&YES&1&\\
$(-5/2,-3/2,1,-1/2)$&$(5,3,1,0)$&$(0,0,0,2)$&NO&&$(5,3,1,1)$\\
$(5/2,1/2,1,3/2)$&$(1,1,0,0)$&$(2,-2,2,5)$&NO&&$(2,0,0,0)$\\
$(-5/2,1/2,1,3/2)$&$(5,1,1,0)$&$(0,2,-2,2)$&NO&&$(5,2,0,0)$\\
$(3/2,5/2,1,1/2)$&$(1,1,0,0)$&$(4,-4,2,1)$&NO&&$(1,0,0,0)$\\
$(3/2,5/2,1,-1/2)$&$(1,1,1,0)$&$(4,-4,0,2)$&NO&&$(2,0,0,0)$\\
$(1/2,3/2,1,5/2)$&$(2,2,0,0)$&$(3,-3,2,3)$&NO&&$(3,1,0,0)$\\
$(1/2,-3/2,1,5/2)$&$(3,2,2,0)$&$(0,3,-3,2)$&NO&&$(7/2,5/2,1/2,1/2)$\\
$(5/2,-1/2,1,-3/2)$&$(2,2,0,0)$&$(1,-1,5,2)$&NO&&$(3,1,0,0)$\\
$(-5/2,-1/2,1,-3/2)$&$(5,2,2,0)$&$(0,1,-1,2)$&NO&&$(5,2,2,1)$\\
$(-3/2,-5/2,1,1/2)$&$(4,4,0,0)$&$(1,-1,2,1)$&NO&&$(9/2,7/2,1/2,1/2)$\\
$(-3/2,-5/2,1,-1/2)$&$(4,4,1,0)$&$(1,-1,1,2)$&NO&&$(4,4,1,1)$\\
$(-1/2,-3/2,1,5/2)$&$(3,3,0,0)$&$(2,-2,3,2)$&NO&&$(4,2,0,0)$\\
$(-1/2,-3/2,1,-5/2)$&$(3,3,3,0)$&$(2,0,-2,2)$&NO&&$(9/2,5/2,3/2,1/2)$\\
&&&&&\\
$(5/2,3/2,-1,1/2)$ &$(2,0,0,0)$&$(0,5,3,1)$&YES&1&\\
$(5/2,3/2,-1,-1/2)$&$(2,1,0,0)$&$(0,0,5,3)$&NO&&\\
$(5/2,-3/2,-1,1/2)$&$(3,2,0,0)$&$(0,0,5,1)$&NO&&\\
$(-5/2,3/2,-1,1/2)$&$(5,2,0,0)$&$(0,0,3,1)$&YES&1&\\
$(5/2,-3/2,-1,1/2)$&$(3,2,0,0)$&$(0,0,5,1)$&NO&\\
$(-5/2,3/2,-1,-1/2)$&$(5,2,1,0)$&$(0,0,0,3)$&NO&&\\
$(-5/2,-3/2,-1,1/2)$&$(5,3,2,0)$&$(0,0,0,1)$&YES&1&\\
$(-5/2,-3/2,-1,-1/2)$&$(5,3,2,1)$&$(0,0,0,0)$&YES&1&\\
$(5/2,1/2,-1,3/2)$&$(2,1,1,0)$&$(0,2,-2,5)$&NO&&\\
$(-5/2,1/2,-1,3/2)$&$(5,2,1,1)$&$(0,0,2,-2)$&NO&&\\
$(3/2,5/2,-1,1/2)$&$(2,1,1,0)$&$(0,4,-4,1)$&NO&&\\
$(3/2,5/2,-1,-1/2)$&$(2,1,1,1)$&$(0,4,0,-4)$&NO&&\\
$(1/2,3/2,-1,5/2)$&$(2,2,2,0)$&$(0,3,-3,3)$&NO&&\\
$(1/2,-3/2,-1,5/2)$&$(3,2,2,2)$&$(0,3,-3,0)$&NO&&\\
$(5/2,-1/2,-1,-3/2)$&$(2,2,2,0)$&$(1,0,-1,5)$&NO&&\\
$(-5/2,-1/2,-1,-3/2)$&$(5,2,2,2)$&$(0,1,0,-1)$&YES&1&\\
$(-3/2,-5/2,-1,1/2)$&$(4,4,2,0)$&$(1,-1,0,1)$&YES&1&\\
$(-3/2,-5/2,-1,-1/2)$&$(4,4,1,0)$&$(1,-1,1,2)$&YES&1&\\
$(-1/2,-3/2,-1,5/2)$&$(3,3,2,0)$&$(2,-2,0,3)$&NO&&\\
$(-1/2,-3/2,-1,-5/2)$&$(3,3,3,2)$&$(2,0,-2,0)$&YES&1&
\end{tabular}

\newpage

%%%%%%%%%%%%%%%%%%%%%%

\section{The case of $GL(n,\bC)$}

The coordinates are the usual ones. The unitary dual of
$GL(n,\bC)$ is known by the results of Vogan \cite{V1}. A representation is unitary if and
only if it is a \textit{Stein complementary series} from a
representation induced from a unitary character on a Levi
component. In order for $\pi$ to have Dirac cohomology, $2\la$ must be
integral. Therefore $\pi$ must be unitarily induced
from a unitary character. 

We note that a unitary character $\bC_\xi$ always has Dirac cohomology equal to $\bC_\xi\otimes Spin$.
Namely, $D=0$ on $\bC_\xi\otimes Spin$. 

In this section we present evidence for the following conjecture. 

\begin{conjecture}\label{conj:gl}
Let $\pi=L(\la,-s\la)$ be an irreducible unitary representation of
$GL(n,\bC)$, such that $2\la$ is regular and integral. Let $\Delta$ be
the positive root system such that $\la$ is dominant, and let $\rho$
be the corresponding half sum of the positive roots. 
Then the Dirac cohomology of $\pi$ is the $K-$type $E(2\la-\rho)$,
with multiplicity $[Spin:E(\rho)]$. 
\end{conjecture}

In the next subsection we prove this conjecture in case $\pi$ is induced
from a unitary character of a maximal parabolic subgroup. We have also 
verified the conjecture in some other cases but we do not present
them because the notation and arguments get increasingly
complicated.
\subsection{Maximal parabolic case}
\label{maxpar}
 Let $G=GL(n,\mathbb{C})$, and let $P=MN$ be a parabolic
  subgroup such that $M=GL(a)\times GL(b)$. We consider the case when
  $\pi$ is induced from a unitary character of $M.$ Then $\la$ is
  formed of integers or half integers. Conjugating
    $2\la$ to be dominant with respect to the usual positive form, we
  can write it as
  \begin{equation}
    \label{eq:glmax1}
2\la=\big(\al+2k,\dots ,\al+2,\al,\al-1,\dots
,\beta+1,\beta,\beta-2,\dots,\beta-2l\big),
  \end{equation}
where $\alpha$ and $\beta$ are integers of opposite
  parity, and $k\ge 0.$
The module is induced from a unitary character on a Levi component
$GL(a)\times GL(b)\subset GL(n=a+b)$. 
Then $a, b$ and the unitary character $\xi=(\xi_1,\xi_2)$ are
\begin{equation}
  \label{eq:glmax2}
  \begin{aligned}
    &a=\frac{\al -\beta+1}{2}+k,&\quad b=\frac{\al-\beta+1}{2}+l,\\
    &\xi_1=\frac{\al+\beta+1}{2}+k,&\quad \xi_2=\frac{\al+\beta-1}{2}-l.
  \end{aligned}
\end{equation}
The case $l<0$ is similar to $l\ge 0,$ so for simplicity of exposition
we treat $l\ge 0$ only.

The condition for $\xi$ to be dominant for the standard positive
system ($\la$ is dominant for it) is that $k+l+1\ge 0.$ By changing
$\la$ to $-\la$ and conjugating to make it dominant, we assume this to
be the case. Assume that $a\ge b$ \ie $k\ge l,$ the other case is similar.

\begin{proposition}
\label{prop:glKtypes}
The $K-$structure of $\pi:=\Ind_P^G [\xi]$ is formed of
\begin{equation*}
  \label{eq:glmax3}
(\xi_1+x_1,\dots ,\xi_1+x_b,\xi_1,\dots ,\xi_1,\xi_2-x_b,\dots,\xi_2-x_1),
\ x_j\in\bN,\ x_i\ge x_{i+1}.
\end{equation*}
occuring with multiplicity 1.
\end{proposition}
\begin{proof}
Change the notation so that $G=U(a,b)$ with maximal compact subgroup
$U(a)\times U(b)$ for this proof only.
The problem of computing the aforementioned multiplicities is
equivalent to computing the multiplicity of the $K-$type
$\mu=\xi_1\otimes \xi_2$ in any finite dimensional representation. A
finite dimensional representation has Langlands parameter given by a
minimal principal series of $G.$ The Levi component is
$M_0=U(1)^{b}\times U(a-b).$ This principal series has to contain
$\mu.$ But $\mu$ is 1-dimensional, and its restriction to $M_0$ is
clear. The result follows from computing the parameter of a
principal series whose Langlands subquotient is finite dimensional
and contains $\mu$. The multiplicity follows from the fact that $\mu$
occurs with multiplicity 1.
\end{proof}

Recall that we need to consider the $K$-type with highest weight
$\tau=2\lambda-\rho$ and try to realize it in the tensor product
$\pi\otimes Spin$, or equivalently in $\pi\otimes E(\rho)$, as a PRV
component.
Since the number of coordinates is $a+b$,
\begin{equation}
  \label{eq:glmax4}
\rho=\bigg(\frac{\al+\beta}{2}+\frac{k+l}{2},\dots ,
-\frac{\al+\beta}{2}-\frac{k+l}{2}\bigg).
\end{equation}
It follows that $\tau$ equals
\begin{equation}
  \label{eq:glmax5}
  \begin{aligned}
\tau=\bigg(&\frac{\beta+\al}{2}+\frac{k-l}{2}+k,\dots ,
\frac{\beta+\al}{2}+\frac{k-l}{2}+1,\\
&\frac{\beta+\al}{2}+\frac{k-l}{2},\dots,
\frac{\beta+\al}{2}+\frac{k-l}{2},\\
&\frac{\beta+\al}{2}+\frac{k-l}{2}-1,\dots,
\frac{\beta+\al}{2}+\frac{k-l}{2}-l\bigg)
  \end{aligned}
\end{equation}
On the other hand, since the $K-$types of $\pi$ have highest 
weight equal to the sum of $\xi$ and roots in $\Delta(\fk n)$, 
$\mu-\rho$ has coordinates
\begin{equation}
  \label{eq:glmax6}
  \begin{aligned}
\mu-\rho=\bigg(
&\frac{2\beta+1}{2}+\frac{k-l}{2}+x_1,\dots,
\frac{\beta+\al}{2}+\frac{k-l}{2}+l+x_b,\\
&\frac{\beta+\al}{2}+\frac{k-l}{2}+l+1+x_{b+1},\dots ,
\frac{\beta+\al}{2}+\frac{k-l}{2}+k+x_a,\\
&\frac{\beta+\al}{2}+\frac{k-l}{2}-l-y_b,\dots ,
\frac{2\al-1}{2}+\frac{k-l}{2}-y_1\bigg)
  \end{aligned}
\end{equation}
with
\begin{equation}
  \label{eq:glmax7}
\begin{aligned}
&\dots \ge x_i\ge x_{i+1}\dots \ge 0\\
&\dots \ge y_j\ge
y_{j+1}\ge \dots \ge 0.
\end{aligned}
\end{equation}
The coordinates in the middle of $\mu-\rho$ are term by term bigger
than the coordinates appear at the beginning of $\tau.$ This forces
$x_{a-k+1}=\dots =x_a=0.$ By the same argument
$y_{b}=\dots=y_{b-l+1}=0$. Note from formula \ref{eq:glmax2} that
$a\ge k$ and $b\ge l.$ The coordinates that are left over from $\tau$
are all equal, so the remaining $y_j,x_i$ are uniquely determined.

\section{Unipotent representations with Dirac cohomology}
\label{unip}
In this section we give an exposition of unipotent representations,
and compute Dirac cohomology for many examples.

\subsection{Langlands Homomorphisms}
In order to explain the parameters of unipotent representations we
recast the classification of $(\fk g,K)-$modules in terms of Langlands
homomorphisms.

First some notation:
%The infinitesimal character of $\pi$ is $(\la_L,\la_R).$
For the field of reals the Weil group is
$$
W_\bR:=\bC^\times\cdot\{1,j\},\quad j^2=-1\in\bC^\times,\quad
jzj^{-1}=\ovl{z},
$$
where  $\Gamma:=Gal(\bC/\bR)$ is the Galois group. There is a canonical
  map $W_\bR\longrightarrow\Gamma$ that maps $\bC^\times$ to $1,$ and
  $j$ to the generator $\gamma$ of $\Gamma.$

\medskip
A linear connected reductive algebraic group $\bb G$ is given by its
 \textit {root datum} $(X^*,R,X_*,\check{R}).$
  ($\bb G\supset\bb B=\bb H\bb N$ where $\bb B$ is a Borel subgroup,
  $\bb H$ a  Cartan subgroup. $X_*$ are the rational characters
  of $\bb H,$ $X^*$ the 1-parameter subgroups, $R$ the roots and $\check
  R$ the coroots).

A real form $G(\bR)=\bb G(\bR)$ is the fixed points of an
  antiholomorphic automorphism $\sig:\bb G(\bC)\longrightarrow \bb
  G(\bC).$ Then $\sig$ induces an automorphism $a$ of the root datum,
  and therefore an automorphism $\ ^\vee a$ of the dual root datum
  $(X_*,\check R,X^*,R)$. The Langlands
 \textit{dual} is $\ ^LG:=\ ^\vee G\ltimes \Gamma$ where $\ ^\vee G$ is the
  complex group attached to the dual root datum. The nontrivial
  element of $\Gamma$ acts on $\ ^\vee G$ by an automorphism induced
  by $\ ^\vee a.$

A \textit{Langlands homomorphism} is a continuous group
  homomorphism $\Phi:W_\bR\longrightarrow \LG,$ satisfying the
  commutative diagram
\[
\begin{matrix}
W_\bR&&\overset{\Phi}{\longrightarrow} &&\
  ^L G\\
&\searrow&&\swarrow& \\
&&\Gamma& &
\end{matrix}
\]
and  such that $\Phi(\bC^\times)$ is formed of semisimple elements.
The main result of the Langlands classification is that $\ ^\vee G$
conjugacy classes of Langlands homomorphisms
  parametrize equivalence classes of irreducible $(\fk g,K)$ modules
(more precisely, characters of $\Phi(W_\bR)/\Phi(W_\bR)_0$).
In the case of a complex group viewed as a real group, this
specializes to the following.

\begin{example} $G(\bR)$ is a complex group $G_0$
viewed as a real group. Then
\[
\vG=\vG_0\times\vG_0, \qquad ^\vee\! a(x,y)=(y,x).
\]
\[
\Phi\longleftrightarrow (\la_L,\la_R)\in\fk h^*\times\fk h^*,\qquad
\la_L-\la_R\in X^*
\]

The irreducible module $L(\la_L,\la_R)$ is obtained as follows. Let $B=HN$ be a
Borel subgroup with $H=T\cdot A$ a Cartan subgroup such that
$T=K\cap H$, and $A$ is split.  Then
$\mu:=\la_L-\la_R$ determines a character of $T,$ $\nu:=\la_L+\la_R$
a character of $A.$  The standard module and irreducible module
attached to $\Phi$ are as before,
\begin{equation*}
\begin{aligned}
&X(\la_L,\la_R):=\Ind_{B}^{G}[\bC_\mu\otimes\bC_\nu\otimes\one]_{K-finite},\\
&L(\la_L,\la_R)\text{ unique irreducible quotient containing } V_\mu.
\end{aligned}
\end{equation*}
\end{example}

\subsection{Unipotent Representations }

An \textit{Arthur parameter} is a homomorphism
$$
\Psi:W_\bR\times SL(2)\longrightarrow \ ^L G
$$
such that $\Psi(W_\bR)$ is bounded. The Langlands homomorphism attached
to $\Psi$ is $\Phi_\Psi(z):=\Psi\bigg(z,
\begin{bmatrix}
  z^{1/2}&0\\0&z^{-1/2}
\end{bmatrix}\bigg)
$

\medskip
A \textit{special unipotent parameter}\ is an Arthur parameter satisfying
$\Psi\mid_{\bC^\times}=Triv.$ Then $\{\Psi\}/\vG $ corresponds
to $\vG$ conjugacy classes $\{\vth,\ve,\vh,\vf\}$
satisfying $\Ad\vth\ve =-\ve,\ \Ad\vth\vh=\vh,\ \Ad\vth\vf=-\vf,$ and
$\Ad\vth^2=Id.$ If we decompose $\vg=\vk+\vs$ according to
the eigenvalues of $\vth,$ then the $\Psi$ are in 1-1 correspondence
with $\vK-$ orbits of nilpotent elements in $\vs.$

\medskip
Fix an infinitesimal character $\chi_{\vO}=\vh/2.$
The \textit{unipotent packet}\ attached to the $\vG-$nilpotent orbit
$\vO$ corresponding to $\Psi$ is  the set of irreducible
representations with annihilator in $U(\fk g)$ maximal containing the
ideal in $\C Z:=U(\fk g)^G$ corresponding to $\chi_{\vO}.$ It is
described in \cite{BV}.

For \textit{general unipotent representations,}\ the same
definitions apply, but each $\vO$ has a finite number of infinitesimal
characters attached to it. For complex classical groups they are
listed in \cite{B}.

\begin{example}
The metaplectic representation of $Sp(2n,\bC)$ is
unipotent. The orbit $\vO\subset so(2n+1,\bC)$ corresponds to the
partition $(2n-1,1,1).$ Then in standard coordinates
\[
\chi_{\vO}=(n-1,\dots ,2,1,1,0)
\]
but the infinitesimal character we want is
\[
(n-1/2,\dots ,1/2).
\]
This is a special case of the procedure to attach finitely many
infinitesimal characters to $\vO.$
\end{example}

\begin{remark}\label{rmk:Ktypes}
In the following we will describe explicitly the unipotent representations with nonzero
Dirac cohomology for each of the types A,B,C,D, and E. After identifying the representations,
we will need a description of their $K-$types. For type A, this follows from Proposition
\ref{prop:glKtypes}. In other cases, the information can be obtained from realizing the unipotent
representations in question via dual pair correspondences. We skip these arguments and simply
state the result in each case.
\end{remark}
\subsection{Type A}\label{sec:A}

Let $G=GL(n,\mathbb{C})$. As we have seen, unipotent $\pi$ correspond to partitions of $n$.
The partition into just one part corresponds to the trivial representation, so we know the Dirac
cohomology is equal to the spin module. In the rest of this section we skip this obvious case.

Since $\lambda$ must be regular, we see from the way $\lambda$ is constructed from a partition that we
should only consider partitions of $n$ into two parts $a,b$ of opposite parity. In particular, $n$ must be odd.
So we take
\begin{equation}
  \label{eq:A.1}
2\la=(a-1,a-3,\dots ,b,b-1,\dots,-b+1,-b,\dots ,-a+3,-a+1),
\end{equation}
where we assume $a>b$.
The corresponding unipotent representation is spherical,
\begin{equation}
  \label{eq:A4}
  \pi=\Ind_{GL(a)\times GL(b)}^{GL(a+b)}[triv\otimes triv].
\end{equation}
Its $K-$structure is formed of
\begin{equation}
  \label{eq:A5}
 \mu= (\al_1,\dots ,\al_b,0,\dots ,0,-\al_b,\dots ,-\al_1),\quad \al_j\in\bN
\end{equation}
occuring with multiplicity 1.
The WF-set has a nilpotent with two columns of length $a$ and $b$.

We see that this case is covered by the results of Subsection \ref{maxpar}.
The Dirac cohomology consists of a single $K$-type with highest weight
\begin{equation}
  \label{eq:A.3}
\tau=(\frac{a-b-1}{2},\dots 1,\unb{2b+1}{0,\dots, 0},-1,\dots ,-\frac{a-b-1}{2})
\end{equation}
and multiplicity $[Spin:E(\rho)]$. The $K$-type $E(\mu)$ of $\pi$ such that $E(\tau)$ appears
in $E(\mu)\otimes E(\rho)$ is given by
\begin{equation}
  \label{eq:A5b}
 \mu= (\frac{a+b-1}{2},\dots ,\frac{a-b-1}{2},0,\dots ,0,-\frac{a-b-1}{2},\dots ,-\frac{a+b-1}{2}).
\end{equation}

\subsection{Type B}\label{sec:B}
Let $G=SO(2n+1,\bC)$. We use the standard coordinates.

To ensure that $2\lambda$ is regular integral, there is
only one possible Arthur parameter, namely the case of $\vO$ equal
to the principal nilpotent orbit. The results in \cite{B}
give more unipotent representations, all spherical. They are associated to the
orbits $\vO$ which have partitions formed of exactly two elements. The WF-set
of a nontrivial unipotent representation $\pi$ must be a nilpotent
with two columns of opposite parity, $2b+1$ and $2a.$ Then $2\la$ is
$W$-conjugate to
\begin{equation}
  \label{eq:B.1}
  (2a,2a-3,\dots ,2;2b-1,2b-3,\dots,1),\qquad a,b\in\bN.
\end{equation}
Only the cases $2b+1>2a$, \ie  $b\geq a$, are unitary. For $b>a$ we get
\begin{equation}
  \label{eq:B.1a}
 2\la= (2b-1,2b-3,\dots,2a+3,2a+1,2a,2a-1,\dots,2,1).
\end{equation}
Since
\begin{equation}
  \label{eq:B.2}
  \rho=(a+b-1/2,a+b-3/2,\dots ,1/2),
\end{equation}
we see that
\begin{equation}
  \label{eq:B3}
  \tau=2\la-\rho=(b-a-1/2,b-a-3/2,\dots,3/2,1/2,1/2,\dots ,1/2),
\end{equation}
with the last $2a+1$ coordinates equal to $1/2$.

The $K-$structure of $\pi$ is
\begin{equation}
  \label{eq:B4}
  (\al_1,\al_1,\al_2,\al_2,\dots ,\al_a,\al_a,\unb{b-a}{0,\dots ,0}),\qquad \al_j\in\bN
\end{equation}
occuring with multiplicity 1.

Now we have to identify $K$-types $E(\mu)$ of $\pi$ such that
$\mu-\rho$ is conjugate to $\tau$ under $W$. We calculate
\begin{multline}
  \label{eq:B4b}
  \mu-\rho=(\al_1-a-b+1/2,\al_1-a-b+3/2,\dots ,\cr
  \al_a+a-b-3/2,\al_a+a-b-1/2,\cr
  \unb{b-a}{a-b+1/2,a-b+3/2,\dots ,-3/2,-1/2}).
\end{multline}
To be conjugate to $\tau$, this expression must have $2a+1$ components equal to $\pm 1/2$. Since there is only one such component among
the last $b-a$ components, the first $2a$ components must all be equal to $\pm 1/2$. Since the first component is smaller than
the second by one, the third component is smaller than the fourth by one, etc., we see that the first, third etc. components must
be $-1/2$ while the second, fourth, etc. components must be $1/2$. This completely determines $\mu$:
\begin{equation}
  \label{eq:B4c}
  \al_1=a+b-1,\,\al_2=a+b-3,\dots ,\,\al_a=b-a+1.
\end{equation}
It is now clear that for this $\mu$ we indeed get a contribution to $H_D(\pi)$, and moreover we can see exactly which $w$
conjugates $\mu-\rho$ to $\tau$.

It remains to consider the case $b=a$. The calculation and the final result are completely analogous. We get
\begin{equation}
  \label{eq:B4d}
  \tau=(1/2,1/2,\dots ,1/2),
\end{equation}
corresponding to
\begin{equation}
  \label{eq:B4e}
  \mu=(2a-1,2a-1,2a-3,2a-3,\dots,1,1).
  \end{equation}

\subsection{Type C}\label{sec:C}
Let $G=Sp(2n,\bC)$. We use the usual coordinates.

As in the other cases, the only Arthur parameters with $2\la$ regular
integral correspond to the principal nilpotent. In this case $\la$
itself is integral. The only other case when $\la$ can be regular
corresponds to the subregular $\vO,$ corresponding to the partition
$1,1,2n-1.$ In this case, the unipotent representations are
the two metaplectic representations, $\pi_{even}$ and $\pi_{odd}$.
The corresponding $\la$ is given by
\begin{equation}
  \label{eq:C.1}
  2\la=(2n-1,2n-3,\dots ,3,1).
\end{equation}
The other cases analogous to type B are \textbf{not} unitary. The
$K-$structures of $\pi_{even}$ and $\pi_{odd}$ are
\begin{equation}\label{C.2}
\begin{aligned}
&(2\al,0,\dots ,0),\\
&(2\al+1,0,\dots 0),\qquad \al\in\bN.
\end{aligned}
\end{equation}
Here $\al=0$ is allowed.
The WF-set is the nilpotent with columns $2n-1,1.$

Since in this case
\begin{equation}
  \label{eq:C.3}
  \rho=(n,n-1,\dots,2,1),
\end{equation}
we see that
\begin{equation}
  \label{eq:C.4}
  \tau=2\la-\rho=(n-1,n-2,\dots,1,0).
\end{equation}
For each of the two metaplectic representations the $K$-types are given by
$\mu=(k,0,0,\dots,0)$, and therefore
\begin{equation}
  \label{eq:C.5}
  \mu-\rho=(k-n,-(n-1),-(n-2),\dots,-2,-1).
\end{equation}
This should be equal to $\tau$ up to $W$, and this happens precisely when $k=n$.
(Recall that $W$ consists of permutations and arbitrary sign changes.)

So we see that for even $n$, $H_D(\pi_{even})$ consists of $E(\tau)$,
for $\tau$ as in (\ref{eq:C.4}), without multiplicity other than the global multiplicity
$[Spin:E(\rho)]$, while $H_D(\pi_{odd})=0$.

For odd $n$, the situation is reversed: $H_D(\pi_{even})=0$, while $H_D(\pi_{odd})$ consists
of $E(\tau)$, with multiplicity $[Spin:E(\rho)]$.

\subsection{Type D}\label{sec:D}
Let $G=SO(2n,\bC)$. We use the usual coordinates.

Since $2\la$ must be regular integral, in this case the
WF-sets of the nontrivial unipotent representations can only be
nilpotents with columns $2b,2a-1,1,$ where $a+b=n$.

By \cite{B}, there are two unipotent (so also unitary)
representations with  $2\la$ $W$-conjugate to $(2a-1,2a-3,\dots
,1;2b-2,\dots ,0)$; the spherical one, and the one with lowest
$K-$type $(1,0,\dots ,0)$ and parameter
\[
(a-1/2,\dots,3/2,-1/2,b-1,\dots ,1,0)\times (a-1/2,\dots ,
3/2,1/2,a-1,\dots ,1,0).
\]
Made dominant for the standard positive system,
\begin{equation}
  \label{eq:D.1}
2\la=(2b-2,2b,\dots,2a+2,2a,2a-1,2a-2,\dots ,1,0).
\end{equation}
(When $b=a,$ the parameter is $(2a-1,2a-2,\dots ,1,0).$) Since 
\begin{equation}
  \label{eq:D.1a}
  \rho= (a+b-1,a+b-2,\dots ,1,0),
\end{equation}
we see that
\begin{equation}
  \label{eq:D.2}
\tau=2\la-\rho=(b-a-1,\dots ,1,0,\unb{2a}{0,\dots ,0})
\end{equation}
The $K-$structure of our unipotent representations is given by
\begin{equation}
  \label{eq:D.3}
  \begin{aligned}
& \mu= (\al_1,\dots ,\al_{2a},0,\dots ,0),\qquad \al_j\in\bN,\
\sum\al_j\in2\bN,\\
&\mu= (\al_1,\dots ,\al_{2a},0,\dots ,0),\qquad \al_j\in\bN,\
\sum\al_j\in2\bN+1.
\end{aligned}
\end{equation}
The first case is for the spherical representation, the second for the
other one. Therefore,
\begin{equation}
  \label{eq:D.4}
 \mu-\rho= (\al_1-(a+b-1),\dots ,\al_{2a}-(b-a),-(b-a-1),\dots ,-1,0).
\end{equation}
Since $\tau$ has $2a+1$ zeros, the only way $\mu-\rho$ can be
conjugate to $\tau$ is to have
\begin{equation}
  \label{eq:D.5}
\al_1=a+b-1,\,\al_2=a+b-2,\dots ,\al_{2a}=b-a.
\end{equation}
Using (\ref{eq:D.3}),
we conclude that for even $a$, the spherical unipotent representation
has $H_D$ equal to $E(\tau)$, with multiplicity $[Spin:E(\rho)]$, while the nonspherical representation
has $H_D=0$. For odd $a$ the situation is reversed: the spherical representation has $H_D=0$, while the
nonspherical one has $H_D$ equal to $E(\tau)$, with multiplicity
$[Spin:E(\rho)]$. 
(Recall that for type D the Weyl group 
consists of permutations combined with an even number
of sign changes, but we can use all sign changes because of the presence of 0.)

\subsection{Type E6}\label{sec:E6}
We use the Bourbaki realization. There are two integral systems,
$A_5A_1$ which gives the nilpotent $3A_1,$ and $D_5T_1$ which gives
$2A_1.$ The parameters are
\begin{equation}
  \label{eq:E6.1}
  \begin{aligned}
&\la=(-5/2,-3/2,-1/2,1/2,3/4,-3/4,-3/4,3/4)\longleftrightarrow 3A_1\\
&\la=(-9/4,-5/4,-1/4,3/4,7/4,-7/4,-7/4,7/4)\longleftrightarrow 2A_1.
  \end{aligned}
\end{equation}
The representations are factors in $\Ind_{A_5}^{E_6}[\bC_\nu].$ The
parameter is
\begin{equation}
  \label{eq:E6.2}
  \begin{aligned}
&(-11/4,-7/4,-3/4,1/4,5/4,-5/4,-5/4,5/4)+\\
+\nu&(1/2,1/2,1/2,1/2,1/2,-1/2,-1/2,1/2).
  \end{aligned}
\end{equation}
The two points above are $\nu=1/2$ and $\nu=1.$ The representations are
unitary because the induced module has multiplicity 1 $K-$structure.
\subsection{Type E7}\label{sec:E7}
We use the Bourbaki realization. There are three integral systems,
$D_6A_1$ which gives the nilpotent $(3A_1)',$  $E_6T_1$ which gives
$2A_1,$ and $A_7$ which gives $4A_1$.  The parameters for the first
two are
\begin{equation}
  \label{eq:E7.1}
  \begin{aligned}
&\la=(0,1,2,3,4,5,-1,1)\longleftrightarrow (3A_1)'\\
&\la=(0,1,2,3,4,-7/2,-17/4,17/4)\longleftrightarrow 2A_1.
  \end{aligned}
\end{equation}
The first representation is a factor in $\Ind_{D_6}^{E_6}[\bC_\nu].$ The
parameter is
\begin{equation}
  \label{eq:E7.2}
  (0,1,2,3,4,5,0,0)+\nu(0,0,0,0,0,0,-1,1).
\end{equation}
The point above is $\nu=1$, an end point of a complementary series. In
any case the representation is multiplicity free, so the
representation is unitary.  The second representation is a factor in
$\Ind_{E_6}^{E_7}[\bC_\nu].$ The parameter is
\begin{equation}
  \label{eq:E7.3}
(0,1,2,3,4-4,-4,4)+\nu(0,0,0,0,0,1,-1/2,1/2)
\end{equation}
with $\nu=1/2.$ The representation is unitary because it is at an end point
complementary series; also the induced module is multiplicity free.

\medskip
The third representation has parameter
\begin{equation}
  \label{eq:E7.4}
  (-9/4,-5/4,-1/4,3/4,7/4,11/4,-4,4).
\end{equation}
This is the minimal length parameter which gives the integral system $A_7.$
By the work of Adams-Huang-Vogan \cite{AHV}, the $K-$structure is multiplicity free
and a full lattice in $E_7$. It does not occur in a multiplicity free
induced module, and is not an end point of a complementary series.

 \subsection{Type E8}\label{sec:E8}
We use the Bourbaki realization. There are two integral systems,
$D_8$ which gives the nilpotent $4A_1,$ and  $E_7A_1$ which gives
$3A_1.$ The parameters are
\begin{equation}
  \label{eq:E8.1}
  \begin{aligned}
&\la=(0,1,2,3,4,5,6,8)\longleftrightarrow 4A_1\\
&\la=(0,1,2,3,4,5,-8,9)\longleftrightarrow 3A_1.
  \end{aligned}
\end{equation}
These are the minimal length parameters which gives the integral
systems $D_8$ and $E_7A_1.$
By the work of Adams-Huang-Vogan \cite{AHV}, the $K-$structure of the first one is multiplicity free
and a full lattice in $E_8$. It does not occur in a multiplicity free
induced module, and is not an end point of a complementary series. The
second one is also multiplicity free, and occurs at an endpoint of a
complementary series. Possibly it is also unitary by an old argument 
of Barbasch-Vogan.


\begin{thebibliography}{9999}

\bibitem [AHV]{AHV} J.~Adams, J.-S.~Huang, D.A.~Vogan, Jr.,
\emph{Functions on the model orbit in $E_8$},
Represent. Theory \textbf{2} (1998), 224--263.

\bibitem [B]{B} D.~Barbasch,
\emph{The unitary dual for complex classical Lie groups}, Invent. Math. \textbf{96} (1989), no. 1, 103--176.

\bibitem [BV]{BV} D.~Barbasch, D.~Vogan,
\emph{Unipotent representations of complex semisimple groups},
Ann. of Math. \textbf{121} (1985), 41--110.

\bibitem [BW]{BW} A.~Borel, N.R.~Wallach,
\emph{Continuous cohomology, discrete subgroups, and representations of reductive groups}, second edition,
Mathematical Surveys and Monographs 67, American Mathematical Society, Providence, RI, 2000.

\bibitem [E]{E} T.~Enright, \emph{Relative Lie algebra cohomology and unitary representations of complex Lie groups},
Duke Math. J. \textbf{46} (1979), no. 3, 513--525.

\bibitem [HKP]{HKP} J.-S.~Huang, Y.-F.~Kang, P.~Pand\v zi\'c, \emph{Dirac
cohomology of some Harish-Chandra modules}, Transform. Groups \textbf{14} (2009), no. 1, 163--173.

\bibitem [HP1]{HP1} J.-S.~Huang, P.~Pand\v zi\'c, \emph{Dirac
cohomology, unitary representations and a proof of a conjecture of
Vogan}, J. Amer. Math. Soc.  \textbf{15} (2002), 185--202.

\bibitem [HP2]{HP2} J.-S.~Huang, P.~Pand\v zi\'c, \emph{Dirac
Operators in Representation Theory}, Mathematics: Theory and
Applications, Birkhauser, 2006.

\bibitem [Kn]{Kn} A.W.~Knapp, \emph{Representation Theory of Semisimple Groups: An Overview Based on Examples},
Princeton University Press, Princeton, 1986. Reprinted: 2001.

\bibitem [K1]{K1} B.~Kostant,
\emph{A formula for the multiplicity of a weight}, Trans. Amer. Math. Soc. \textbf{93} (1959), 53--73.

\bibitem [K2]{K2} B.~Kostant,
\emph{Lie algebra cohomology and the generalized Borel-Weil
theorem}, Ann. of Math. \textbf{74} (1961), 329--387.

\bibitem [LS]{LS} J.-S.~Li, J.~Schwermer, \emph{Automorphic representations and cohomology of arithmetic groups},
Challenges for the 21st century (Singapore, 2000), 102--137, World Sci. Publ., River Edge, NJ, 2001.

\bibitem [P1]{P1} R.~Parthasarathy, \emph{Dirac operator and the discrete
series}, Ann. of Math. \textbf{96} (1972), 1--30.

\bibitem [P2]{P2} R.~Parthasarathy, \emph{Criteria for the unitarizability of some highest weight modules},
Proc. Indian Acad. Sci. \textbf{89} (1980), 1--24.

\bibitem [PRV]{PRV} K.R.~Parthasarathy, R.~Ranga Rao, V.S.~Varadarajan, \emph{Representations
of complex semisimple Lie groups and Lie algebras}, Ann. of Math. \textbf{85} (1967), 383--429.

\bibitem [SR]{SR}  S.A.~Salamanca-Riba,
\emph{On the unitary dual of real reductive Lie groups and the $A_{\fk q}(\lambda)$ modules: the strongly regular
case}, Duke Math. J. \textbf{96} (1998), 521--546.

\bibitem [V1]{V1} D.A.~Vogan, Jr., \emph{The unitary dual of ${\rm GL}(n)$ over an Archimedean field},
Invent. Math. \textbf{83} (1986), no. 3, 449--505.

\bibitem [V2]{V} D.A.~Vogan, Jr., \emph{Dirac operators and unitary
representations}, 3 talks at MIT Lie groups seminar, Fall 1997.

\bibitem [VZ]{VZ} D.A.~Vogan, Jr., and G.J.~Zuckerman,
\emph{Unitary representations with non-zero cohomology}, Comp. Math. \textbf{53} (1984), 51--90.

\bibitem [W]{Warner} G.~Warner,
\emph{Harmonic analysis on semisimple Lie groups I}, Springer-Verlag, Berlin, Heidelberg, New York, 1972.

\bibitem [Zh]{Zh} D.P.~Zhelobenko,
\emph{Harmonic analysis on complex semisimple Lie groups}, Mir, Moscow, 1974.

\end{thebibliography}

\begin{thebibliography}{A}

\bibitem [A]{A} T. Aoki, \textit{Calcul exponentiel des op\'erateurs
microdifferentiels d'ordre infini.} I, Ann. Inst. Fourier (Grenoble)
\textbf{33} (1983), 227--250.

\bibitem [B]{B} R. Brown, \textit{On a conjecture of Dirichlet},
Amer. Math. Soc., Providence, RI, 1993.

\bibitem [D]{D} R. A. DeVore, \textit{Approximation of functions},
Proc. Sympos. Appl. Math., vol. 36,
Amer. Math. Soc., Providence, RI, 1986, pp. 34--56.

\end{thebibliography}
\end{document}